\documentclass[runningheads]{llncs}

\usepackage{graphicx}
\usepackage{multirow}%

\usepackage{amsmath,amssymb,amsfonts}%
\usepackage{amsthm}%
\usepackage{mathrsfs}%
\usepackage[title]{appendix}%
\usepackage{xcolor}%
\usepackage{textcomp}%
\usepackage{manyfoot}%
\usepackage{booktabs}%
\usepackage{algorithm}%
\usepackage{algorithmicx}%
\usepackage{algpseudocode}%
\usepackage{listings}%
\usepackage{indentfirst}
\usepackage{enumerate}
\usepackage[all]{xy}
\usepackage{marvosym}
\begin{document}
\title{Endomorphism Rings of Supersingular Elliptic Curves and Ternary Quadratic Forms}
\titlerunning{Supersingular Elliptic Curves and Quadratic Forms}
%
\author{Guanju Xiao\inst{1} \and
Zijian Zhou\inst{2} \and
Longjiang Qu\inst{2}\textsuperscript{(\Letter)}}
\authorrunning{G. Xiao, Z. Zhou and L. Qu}
\institute{ Hubei Provincial Engineering Research Center of Intelligent Connected Vehicle Network Security, School of Cyber Science and Technology, Hubei University, Wuhan 430062, China\\
 \email{gjxiao@amss.ac.cn} \and
College of Science, National University of Defense Technology, Changsha 410073, China
\\
\email{zhouzijian122006@163.com;ljqu\_happy@hotmail.com}}
\maketitle              
\begin{abstract}

Let $c<3p/16$ be a prime or $c=1$. Let $E$ be a $\mathbb{Z}[\sqrt{-cp}]$-oriented supersingular elliptic curve defined over $\mathbb{F}_{p^2}$. There exists a $c$-isogeny from $E$ to $E^p$ with kernel $G \subset E[c]$. Given an Eichler order corresponding to the endomorphism ring $\text{End}(E,G)=\{ \theta \in \text{End}(E): \theta(G) \subseteq G \}$, we can compute a ternary quadratic form with discriminant $p$ by solving two square roots in $\mathbb{F}_c$, and the ternary quadratic form corresponds to a maximal order $\mathcal{O} \cong \text{End}(E)$ in $B_{p,\infty}$ by Brandt--Sohn correspondence.
Let $D$ be a prime with $D<p$ (resp. $4D<p$). If an imaginary quadratic order with discriminant $-D$ (resp. $-4D$) can be embedded into $\text{End}(E)$, then we can compute a maximal order in $B_{p,\infty}$ corresponding to $\text{End}(E)$ by solving one square root in $\mathbb{F}_D$ and two square roots in $\mathbb{F}_c$.

As we know, any isogeny between supersingular elliptic curves can be translated into a kernel ideal of the endomorphism ring. We study the action of the kernel ideal and give a basis of its right order. In general, we propose an efficient algorithm for computing a maximal order from an Eichler order in $B_{p,\infty}$.

\end{abstract}
\section{Introduction}\label{sec1}
Given a supersingular elliptic curve, the supersingular endomorphism ring problem is to compute all of its endomorphisms, which means to compute a basis of the endomorphism ring. The supersingular endomorphism ring problem is not only a challenging problem in computational number theory, but also plays a key role in isogeny-based cryptography. In isogeny-based cryptography, one typically uses supersingular elliptic curves and isogenies between them to construct cryptographic schemes. Assuming the Generalised Riemann Hypothesis, the supersingular endomorphism ring problem is equivalent to the $\ell$-isogeny path problem (see \cite{MR3794837} or \cite{9719728}).

A key exchange protocol called SIDH, based on the path finding problem in the supersingular isogeny graph, was proposed in \cite{MR2931459}. However, SIDH reveals information about the images of certain torsion points.
In 2023, SIDH has been broken by using the torsion point information (see \cite{MR4591003,MR4591004,MR4591005}). Until now, there do not exist efficient algorithms to solve either the path finding problem or the supersingular endomorphism ring problem. These problems are closely related to other cryptographic protocols such as CSIDH \cite{MR3897883} and SQISign \cite{SQISign}.

In 1941, Deuring \cite{MR5125} proved that there is a one-to-one correspondence between type classes of maximal orders in $B_{p,\infty}$ and isomorphism classes of supersingular elliptic curves up to the action of $\text{Gal}(\mathbb{F}_{p^2}/\mathbb{F}_p)$, where $B_{p,\infty}$ is the unique quaternion algebra over $\mathbb{Q}$ ramified only at $p$ and $\infty$. Computing the endomorphism ring of a supersingular elliptic curve was first studied by Kohel \cite{MR2695524}, who gave a probabilistic algorithm for generating a subring of finite index of the endomorphism ring in time $O(p^{1+\varepsilon})$. Following this way, Eisentr\"{a}ger et al. \cite{Eisentrager2020} gave an algorithm for computing the endomorphism ring of a supersingular elliptic curve defined over $\mathbb{F}_{p^2}$ that runs, under certain heuristics, in time $O((\log p)^2 p^{1/2})$. Moreover, Page and Wesolowski \cite{eprint2023/1399} proposed an unconditional probabilistic algorithm to compute the endomorphism ring in time $\tilde{O}(p^{1/2})$.

Brandt and Sohn \cite{MR0017775,Sohn} proved that there exists an explicit bijection between equivalent classes of ternary quadratic forms of discriminant $p$ and type classes of maximal orders in $B_{p,\infty}$. This bijection is also called Brandt--Sohn correspondence. By this correspondence, Cervi\~{n}o \cite{cervino2004} proposed a deterministic and explicit algorithm to compute endomorphism rings of supersingular elliptic curves. In this paper, we further study relations between orientations of supersingular elliptic curves and coefficients of ternary quadratic forms in this correspondence.

Let $E$ be a $\mathbb{Z}[\sqrt{-cp}]$-oriented supersingular elliptic curve. If $c=1$, then the curve $E$ is defined over $\mathbb{F}_p$. Moreover, there exists a one-to-one correspondence between $\mathbb{F}_p$-isomorphism classes of supersingular elliptic curves and primitive reduced binary quadratic forms with discriminant $-p$ or $-16p$ \cite{xiao2022endomorphism}. If $c< 3p/16$ is a prime, then there exists a $c$-isogeny $\phi:E\to E^p$ with $\ker(\phi)=G \subseteq E[c]$. Define $\text{End}(E,G):=\{\theta \in \operatorname{End}(E): \theta(G) \subseteq G\}\subseteq \text{End}(E)$, and the authors \cite{xiao2023oriented} proved that $\text{End}(E,G)$ is isomorphic to an Eichler order $\mathcal{O}_c(q,r)$ or $\mathcal{O}'_c(q,r')$. Moreover, any Eichler order $\mathcal{O}_c(q,r)$ (resp. $\mathcal{O}'_c(q,r')$) can represent a binary quadratic form with discriminant $-16cp$ (resp. $-cp$).

%
In this paper, for any binary quadratic form with discriminant $-16cp$, we show that it can be represented by a unique reduced ternary quadratic form by solving two square roots in $\mathbb{F}_c$. Moreover, we prove that the reciprocal of this ternary quadratic form (definition in Section 3) has discriminant $p$, which corresponds to a maximal order in $B_{p,\infty}$ by Brandt--Sohn correspondence.

Let $D$ be a prime with $D<p$ (resp. $4D<p$). If an imaginary quadratic order $O$ with discriminant $-D$ (resp. $-4D$) can be embedded into $\text{End}(E)$, then we can compute a unique binary quadratic form with discriminant $-16cp$ or $-cp$ by solving one square root in $\mathbb{F}_D$. Moreover, we can compute a ternary quadratic form with discriminant $p$ by solving two square roots in $\mathbb{F}_c$, and the ternary quadratic form corresponds to $\text{End}(E)$ by Brandt--Sohn correspondence.


We also study isogenies between supersingular elliptic curves. Let $\ell \neq p$ be a prime. We divide into two cases.
\begin{itemize}
  \item Let $\varphi: E \to E'$ be a $\mathbb{Z}[\sqrt{-cp}]$-oriented $\ell$-isogeny. This isogeny can be represented by a binary quadratic form with discriminant $-16cp$. Given the endomorphism ring $\text{End}(E,G)$,  we can compute $\text{End}(E', \varphi(G))$ by \cite[Theorems 5 and 6]{xiao2023oriented}. Moreover, we can compute a ternary quadratic form with discriminant $p$ which corresponds to $\text{End}(E')$.
  \item If the $\ell$-isogeny $\varphi: E \to E'$ is not $\mathbb{Z}[\sqrt{-cp}]$-oriented, then we give the kernel ideal of $\varphi$  in Section 5.2. We also represent the right order of this kernel ideal by a ternary quadratic form, and give the endomorphism ring of $E'$. Moreover, we propose an efficient algorithm for computing maximal orders from Eichler orders in $B_{p,\infty}$.
\end{itemize}

There are many algorithms for solving square roots in finite fields, and the one proposed by Tonelli and Shanks is a quite efficient algorithm with a polynomial time under the Generalized Riemann Hypothesis \cite[\S 1.5]{MR1228206}. Solving square roots in $\mathbb{F}_D$ or $\mathbb{F}_c$ is in polynomial time of $O(\log p)$ since $c <p$ and $D<p$ in this paper.

\textbf{This paper is organized as follows.} In Section 2, we review some preliminaries on endomorphism rings of supersingular elliptic curves and ternary quadratic forms. We study the representation of binary quadratic forms by ternary quadratic forms in Section 3. We compute the endomorphism ring of any oriented supersingular elliptic curve by Brandt-Sohn correspondence in Section 4. In Section 5, we study the action of $\ell$-isogenies and compute endomorphism rings of isogenous curves. Finally, we make a conclusion in Section 6.

\section{Preliminaries}\label{sec2}
In this section, we introduce some preliminaries on supersingular elliptic curves over finite fields and their endomorphism rings, quaternion orders and ternary quadratic forms.
\subsection{Elliptic curves over finite fields}
We present some basic facts about elliptic curves over finite fields, and the readers can refer to \cite{MR2514094} for more details.

Let $\mathbb{F}_{p^k}$ be a finite field with characteristic $p>3$. An elliptic curve $E$ defined over $\mathbb{F}_{p^k}$ can be written as a Weierstrass model $E:Y^2=X^3+aX+b$ where $a,b \in \mathbb{F}_{p^k}$ with $4a^3+27b^2 \neq 0$. The $j$-invariant of $E$ is $j(E)=1728\cdot 4a^3/(4a^3+27b^2)$. Different elliptic curves with the same $j$-invariant are isomorphic over $\overline{\mathbb{F}}_{p^k}$. The chord-and-tangent addition law makes of $E(\mathbb{F}_{p^k})=\{(x,y)\in \mathbb{F}_{p^k} ^2:y^2=x^3+ax+b\} \cup \left\{\infty \right\}$ an abelian group.

Let $E_1$ and $E_2$ be elliptic curves defined over $\mathbb{F}_{p^k}$. An isogeny from $E_1$ to $E_2$ is a morphism $\phi:E_1 \rightarrow E_2$ satisfying $\phi(\infty)=\infty$. We always assume $\phi \neq 0$. The isogeny $\phi$ is a surjective group homomorphism with finite kernel, and it is called a $\mathbb{F}_{p^k}$-isogeny if it is defined over $\mathbb{F}_{p^k}$. The degree of an isogeny $\phi$, written as $\deg(\phi)$, is its degree as a rational map. If $\phi$ is separable, we have $\deg(\phi)=\# \ker(\phi)$. Note that all isogenies of prime degree $\ell\neq p$ are separable.

An endomorphism of $E$ is an isogeny from $E$ to itself. The Frobenius map $\pi :(x,y) \mapsto (x^{p^k},y^{p^k})$ is an inseparable endomorphism. The characteristic polynomial of $\pi$ is $x^2-tx+p^k$, where $t$ is the trace of $\pi$. It is well known that $E$ is supersingular (resp. ordinary) if $p\mid t$ (resp. $p \nmid t$). Moreover, the $j$-invariant of every supersingular elliptic curve over $\overline{\mathbb{F}}_p$ is proved to be in $\mathbb{F}_{p^2}$ and it is called a supersingular $j$-invariant.

The set $\text{End}(E)$ of all endomorphisms of $E$ (along with the zero map) over $\overline{\mathbb{F}}_{p^k}$ form a ring under the usual addition and composition as multiplication. The set $\text{End}_{\mathbb{F}_{p^k}}(E)$ is a subring of $\text{End}(E)$ which contains all the endomorphisms over $\mathbb{F}_{p^k}$.

\subsection{Endomorphism ring and quaternion algebra}
For a supersingular elliptic curve $E$ over $\mathbb{F}_{p^2}$, Deuring \cite{MR5125} has proved that the endomorphism ring $\text{End}(E)$ is isomorphic to a maximal order in $B_{p,\infty}$, where $B_{p,\infty}$ is the unique quaternion algebra over $\mathbb{Q}$ ramified only at $p$ and $\infty$.

We present some basic facts about quaternion algebras, and the readers can refer to \cite{MR4279905} for more details. Let $\bar{}:B_{p,\infty} \to B_{p,\infty}$ be the standard involution. For any $\alpha \in B_{p,\infty}$, we define the reduced trace $\text{trd}(\alpha)=\alpha+\bar{\alpha}$ and reduced norm $\text{nrd}(\alpha)=\alpha \bar{\alpha}$ (see \cite[Section 3.1]{MR4279905}).

An order $\mathcal{O}$ of $B_{p,\infty}$ is a subring of $B_{p,\infty}$ which is also a $\mathbb{Z}$-lattice of rank 4. Given a basis $\langle \alpha_1,\alpha_2,\alpha_3,\alpha_4 \rangle$, the discriminant of $\mathcal{O}$ is defined as $\text{disc}(\mathcal{O})=\text{det}(\text{trd}(\alpha_i \bar{\alpha}_j))_{i,j \in \{1,2,3,4\}}$.
We have that the discriminant of an order is an integer and is independent of the choice of basis.

Two orders $\mathcal{O}_1$ and $\mathcal{O}_2$ are in the same type if and only if there exists an element $\alpha \in B_{p,\infty}^{\times}$ such that $\mathcal{O}_1=\alpha^{-1} \mathcal{O}_2 \alpha$. An order is called maximal if it is not properly contained in any other order. The discriminant of a maximal order in $B_{p,\infty}$ is $p^2$.

An Eichler order in a quaternion algebra is the intersection of two (not necessarily distinct) maximal orders. The level of an Eichler order in $B_{p,\infty}$ is given by its index in one of the maximal orders whose intersection defines the order (the index will be the same for either order). Notice that an Eichler order of level $1$ is a maximal order.

Let $\mathcal{O}$ be an Eichler order and $I$ an invertible left ideal of $\mathcal{O}$.
Define the left order $\mathcal{O}_L(I)$ and the right order $\mathcal{O}_R(I)$ of $I$ by
$$\mathcal{O}_L(I)=\left\{x\in B_{p,\infty}:xI\subseteq I \right \}, \quad \mathcal{O}_R(I)=\left\{x\in B_{p,\infty}:Ix \subseteq I\right \}  .$$
Moreover, $\mathcal{O}_L(I)=\mathcal{O}$, and $\mathcal{O}_R(I)=\mathcal{O}^{'}$ is also an Eichler order of the same level, in which case we say that $I$ connects $\mathcal{O}$ and $\mathcal{O}^{'}$. The inverse of an ideal is given by $I^{-1}=\{ a \in B_{p,\infty} : IaI \subseteq I\}$. In particular, we have $II^{-1}=\mathcal{O}_{L}(I)$ and $I^{-1}I=\mathcal{O}_R(I)$. The reduced norm of $I$ can be defined as $$\text{nrd}(I)=\text{gcd}(\{\text{nrd}(\alpha) : \alpha \in I \}).$$

Deuring \cite{MR5125} gave an equivalence of categories between supersingular $j$-invariants and maximal orders in the quaternion algebra $B_{p,\infty}$. Furthermore, if $E$ is a supersingular elliptic curve with $\text{End}(E)\cong \mathcal{O}$, then there is a one-to-one correspondence between isogenies $\phi:E\to E'$ and left $\mathcal{O}$-ideals $I$. More details on the correspondence can be found in \cite[Chapter 42]{MR4279905}.

If the supersingular elliptic curve $E$ is defined over $\mathbb{F}_p$, then $\text{End}(E)$ contains an element $\alpha$ with minimal polynomial $x^2+p$ (see \cite{MR3451433}). Equivalently, the endomorphism ring $\text{End}(E)$ contains a subring $\mathbb{Z}[\sqrt{-p}]$. Ibukiyama \cite{MR683249} has given an explicit description of all maximal orders $\mathcal{O}$ in $B_{p, \infty}$ containing a root of $x^2+p=0$.

\subsection{Quaternion orders and ternary quadratic forms}
Brandt \cite{MR0017775} constructed maximal orders in the quaternion algebra from ternary lattices via even Clifford algebras. The idea was then exploited by Sohn \cite{Sohn} in his thesis where he proved the following proposition.
\begin{proposition}(\cite[Satz 5.1]{Sohn})\label{BS1}
There exists a bijection between the classes of ternary quadratic forms of discriminant $p$ and the type classes of maximal orders in the quaternion algebra $B_{p,\infty}$.
\end{proposition}
This bijection is called Brandt--Sohn correspondence. We introduce the explicit bijection between ternary quadratic forms and quaternion orders. The reader can refer to \cite[Chapter 22]{MR4279905} for more details.

Let $Q(x,y,z)=2ax^2+2by^2+2cz^2+2uyz+2vxz+2wxy\in \mathbb{Z}[x,y,z]$ be a ternary quadratic form. It can be represented by the following matrix:
$$A_Q=
 \begin{bmatrix}
   2a & w & v \\
   w & 2b & u \\
   v & u & 2c
 \end{bmatrix}
 .$$
The discriminant of $Q$ is $\text{disc}(Q)=\det (A_Q) /2=4abc+uvw-au^2-bv^2-cw^2=N$. We define the discriminant of a ternary quadratic form just as \cite{Dickson1930}, and this is different in \cite[Chapter 22]{MR4279905}. We associate the quaternion $\mathbb{Z}$-order $\mathcal{O} \subseteq B_{p, \infty}$ with basis $1$, $i$, $j$, $k$ and multiplication laws
\begin{align}\label{equation1}
 \begin{array}{ll}
  i^2=ui-bc \quad \quad & jk=a\bar{i}=a(u-i) \vspace{0.5ex}  \\
  j^2=vj-ac \quad \quad & ki=b\bar{j}=b(v-j) \vspace{0.5ex}  \\
  k^2=wk-ab \quad \quad & ij=c\bar{k}=c(w-k).
 \end{array}
\end{align}
The other multiplication rules are determined by the skew commutativity relations coming from the standard involution.

The order $\mathcal{O}=\mathbb{Z}+\mathbb{Z}i+\mathbb{Z}j+\mathbb{Z}k$ defined by (\ref{equation1}) is called the even Clifford algebra $\text{Clf}^0(Q)$ of $Q$, and the discriminant of $\mathcal{O}$ is $\text{disc}(\mathcal{O})=N^2$. At least one of the minors
$$u^2-4bc, \quad v^2-4ac, \quad w^2-4ab$$
is nonzero since $Q$ is nondegenerate, so for example if $w^2-4ab \neq 0$, completing the square we find
$$\mathcal{O} \subset B_{p,\infty} \cong \left( \frac{w^2-4ab,-aN}{\mathbb{Q}} \right).$$

We define an inverse to the even Clifford algebra construction. Let $\mathcal{O} \subset B_{p,\infty}$ be a quaternion order with discriminant $\text{disc}(\mathcal{O})=N^2$. As we know, the order $\mathcal{O}$ is a lattice in $B_{p, \infty}$. Let $\{ 1,i,j,k\}$ be a $\mathbb{Z}$-basis of $\mathcal{O}$. We define the dual of $\mathcal{O}$ (with respect to trd) to be $\mathcal{O}^{\#}:=\{ \alpha \in B_{p,\infty}: \text{trd}(\alpha \mathcal{O}) \subseteq \mathbb{Z} \}=\{ \alpha \in B_{p,\infty}: \text{trd}(\mathcal{O} \alpha) \subseteq \mathbb{Z} \}$. Let $\{ 1^{\#}, i^{\#}, j^{\#}, k^{\#} \}$ be the dual basis. We have $\text{trd}(1 \cdot i^{\#})=\text{trd}(1 \cdot j^{\#})=\text{trd}(1 \cdot k^{\#})=0$ (see \cite[15.6.3]{MR4279905}). Let
$$(\mathcal{O}^{\#})^0=\{ \alpha \in \mathcal{O}^{\#}: \text{trd}(\alpha)=0\}=\mathbb{Z} i^{\#}+\mathbb{Z} j^{\#}+\mathbb{Z} k^{\#}$$
be the trace zero elements in the dual of $\mathcal{O}$ with respect to the reduced trace pairing.
We have
$$N i^{\#}=jk-kj, \quad N j^{\#}=ki-ik, \quad N k^{\#}=ij-ji.$$
Then the following map
$$\begin{array}{rcl}
  N\text{nrd}^{\#}(\mathcal{O}):(\mathcal{O}^{\#})^0 & \to     & \mathbb{Z} \\
                         \alpha & \mapsto & N \cdot \text{nrd}(\alpha)
\end{array} $$
defines a candidate quadratic form
$$N\text{nrd}^{\#}(x,y,z):=N\text{nrd}^{\#}(xi^{\#}+yj^{\#}+zk^{\#})=\text{nrd}(x(jk-kj)+y(ki-ik)+z(ij-ji)).$$
One can verify that $N\text{nrd}^{\#}(x,y,z)$ has discriminant $N$.

%
%
%
%

\section{Representation of binary quadratic forms by ternary quadratic forms}

We first present some basic facts about ternary quadratic forms (see \cite[Chapter 1]{Dickson1930}), and then study the representation of binary quadratic forms by ternary quadratic forms.

Let $f$ be a ternary quadratic form with integer coefficients, given by the equation
$$f(x,y,z)=a_{11}x^2+a_{22}y^2+a_{33}z^2+2a_{23}yz+2a_{13}xz+2a_{12}xy.$$

Define the matrix of $f$ to be
 $$A=A_f=
 \begin{bmatrix}
   a_{11} & a_{12} & a_{13} \\
   a_{12} & a_{22} & a_{23} \\
   a_{13} & a_{23} & a_{33}
 \end{bmatrix}
 .$$
Define the discriminant of $f$ to be $\text{disc}(f)=\det(A)/2$.

Let $\tau$ denote the gcd (greatest common divisor) of the elements $a_{11}$, $a_{22}$, $a_{33}$, $a_{23}$, $a_{13}$, $a_{12}$ in the matrix $A$. Let $\sigma$ denote the gcd of the coefficients $a_{11}$, $a_{22}$, $a_{33}$, $2a_{23}$, $2a_{13}$, $2a_{12}$. If $\tau=1$, then $f$ is called primitive. If $\sigma=\tau=1$, then $f$ is called properly primitive. If $\tau=1$ and $\sigma=2$, then $f$ is called improperly primitive.

Let $A_{i,j}$ be the $i,j$-cofactor of $A$. That is,
\begin{align*}
  A_{11}=a_{22}a_{33}-a_{23}^2, \quad & A_{23}=a_{13}a_{12}-a_{11}a_{23}=A_{32}, \\
  A_{22}=a_{11}a_{33}-a_{13}^2, \quad & A_{13}=a_{23}a_{12}-a_{13}a_{22}=A_{31}, \\
  A_{33}=a_{11}a_{11}-a_{12}^2, \quad & A_{12}=a_{13}a_{23}-a_{12}a_{33}=A_{21}.
\end{align*}
Define the divisor of $f$ to be the positive integer
$$\Omega=\Omega_f=\gcd(A_{11},A_{22},A_{33},A_{23},A_{13},A_{12}),$$
and $\Omega$ is an invariant of $f$.
We write $F'=\Sigma_{i,j} A_{i,j} X_i X_j=\Omega F$ and call $F$ the reciprocal of $f$. Evidently, $F$ is a primitive form.

Assume that $f$ is primitive. One can check that the reciprocal of $F$ is $f$. Moreover, we have that $\Omega^2$ divides all $\det(A) \cdot a_{i,j}$ and hence divides $\det(A)$. So $\det(A)=\Delta \Omega^2$ define an integer $\Delta$ which is an invariant of $f$.

Two ternary quadratic forms $f$ and $g$ are said to be equivalent, $f \sim g$, if there exists a unimodular matrix $U=(u_{ij})_{3\times 3}$ such that $A_g=UA_f U^{\mathrm{t}}$. (That is, $U$ has integer entries and $\det(U)=\pm 1$; $U^{\textrm{t}}$ is its transpose.) Moreover, we have the following proposition.

\begin{proposition}(\cite[Theorem 3]{Dickson1930})
If two ternary quadratic forms are equivalent, then their reciprocals are equivalent as well.
\end{proposition}


The ternary quadratic form $f$ is called positive definite if $f(x,y,z)>0$ for real numbers $x,y,z$ unless $x=y=z=0$.
\begin{proposition}(\cite[Theorem 8]{Dickson1930})\label{Dick1}
If $f$ is a positive ternary quadratic form, then $\operatorname{disc}(f)$, $a_{ii}$ and $A_{ii}$ ($i=1,2,3$) are all positive. If $a_{11}$, $A_{33}$ and $\operatorname{disc}(f)$ are positive, then the ternary quadratic form is positive.
\end{proposition}

We begin to discuss the representation of binary quadratic forms by ternary quadratic forms.

\begin{proposition}(\cite[Theorem 36]{Dickson1930})\label{Dick5}
Let $\rho=(a,2t,b)$ be a positive primitive binary quadratic form of discriminant $(2t)^2-4ab=-4\Omega C < 0$ with $\Omega >0$. Let $\Delta >0$ be an odd integer. Assume that the congruences
\begin{align}\label{equation2}
R^2+\Delta a \equiv 0, \quad RS-\Delta t \equiv 0, \quad S^2+\Delta b \equiv 0 \large \pmod C
\end{align}
are solvable. If $C$ and $\Delta$ are relatively prime, then $\rho$ can be represented by a positive properly primitive ternary quadratic form $f$ with invariants $\Omega$ and $\Delta$.

In particular, assume that $\Omega$ is even and $\rho_1$ is any odd integer represented by $\rho$. If $C$ is double an odd prime or $C=2$, and $\Delta \rho_1 \equiv 3 \pmod 4$ (or if $C=4$ and $\Delta \rho_1 \equiv 7 \pmod 8$), then the reciprocal $F$ of $f$ is improperly primitive.
\end{proposition}

\begin{proof}
Dickson gave a proof of this theorem but the last sentence. We also follow notations in \cite[\S 21]{Dickson1930}. Assume that congruences (\ref{equation2}) are solvable. We can assume that $A$, $B$ and $T$ satisfy the following equations:
\begin{align}\label{equation3}
BC-R^2=\Delta a, \quad RS-TC=\Delta t, \quad AC-S^2=\Delta b.
\end{align}
Moreover, the integers $r$, $s$ and $c$ satisfy the following equations:
\begin{align}\label{equation4}
ST-AR=\Delta r, \quad RT-BS=\Delta s, \quad AB-T^2=\Delta c.
\end{align}
The binary quadratic form $\rho=(a,2t,b)$ can be represented by a properly primitive ternary quadratic form $f(x_1,x_2.x_3)=ax^2_{1}+bx^2_{2}+cx^2_{3}+2rx_{2}x_3+2sx_{1}x_3+2tx_1x_{2}$.
Because $C$ and $\Delta$ are relatively prime, the reciprocal $F$ of $f$ is  $$F(X_1,X_2,X_3)=AX^2_{1}+BX^2_{2}+CX^2_{3}+2RX_{2}X_3+2SX_{1}X_3+2TX_{1}X_2.$$

In our case, $C$ is even, so $F$ is improperly primitive if and only if $A$ and $B$ are even.

If $C=4=2\times 2$, then $t^2-ab=-4 \Omega \equiv 0 \pmod {8}$. We can assume $2 \nmid b$. We have
$AC-S^2 \equiv \Delta b \equiv -S^2 \pmod 4$, so $2 \nmid S$. Moreover, since $S^2 \equiv 1 \pmod 8$ and $\Delta b \equiv 7 \pmod 8$ ($b$ can be represented by $\rho$), we have $AC \equiv S^2+\Delta b \equiv 0 \pmod 8$. It follows $2 \mid A$.

If $2 \nmid a$ and $2 \nmid t$, we can prove $2 \mid B$.
If $2 \parallel t$, then $4 \parallel a$. We have $RS-TC =\Delta t \equiv 2 \pmod 4$, so $RS \equiv 2 \pmod 4$. It follows $2 \parallel R$. We have $BC=R^2 +\Delta a \equiv R^2 +4 \equiv 0 \pmod 8$. If $4 \mid t$, then $8 \mid a$. We have $RS=TC +\Delta t \equiv 0 \pmod 4$, so $4 \mid R$. It follows $BC =R^2+\Delta a \equiv 0 \pmod 8$. So we have $2 \mid B$.

If $C$ is double an odd prime or $C=2$, then $t^2-ab =-\Omega C \equiv 0 \pmod 4$. We can assume $2 \nmid b$ and $\Delta b \equiv 3 \pmod 4$. We have $AC=S^2+\Delta b \equiv 0 \pmod 2$, so $ 2\nmid S$. It follows that $AC \equiv 0 \pmod 4$, which implies $2 \mid A$. Similarly, we can show $2 \mid B$ if $a$ is odd. If $a$ is even, then $2 \mid t$ and $4 \mid a$. We have $BC=R^2 + \Delta a \equiv 0 \pmod 2$, which implies that $2 \mid R$ and $BC \equiv 0 \pmod 4$, we also have $2 \mid B$.
\end{proof}

\begin{proposition}\label{Dick6}
Let $\rho=(a,2t,b)$ be a positive binary quadratic form with discriminant $(2t)^2-4ab=-4\Omega C < 0$ which satisfies $2 \parallel t$ and $\gcd(a,2t,b)=4$. Assume that $C$ is double an odd prime or $C=2$ and $\Omega >0$ is even. Let $\Delta >0$ be an odd integer. Assume that congruences $(\ref{equation1})$ are solvable. If $C$ and $\Delta$ are relatively prime, then $\rho$ can be represented properly by a positive properly primitive ternary quadratic form $f$ with invariants $\Omega$ and $\Delta$. Moreover, the reciprocal $F$ of $f$ is improperly primitive.
\end{proposition}
\begin{proof}
Similarly, the form $\rho$ can be represented properly by a positive ternary form $f$.

Following notations in equations (\ref{equation2}), (\ref{equation3}) and (\ref{equation4}).  We have $BC=R^2 + \Delta a \equiv 0 \pmod 2$. It follows that $2 \mid R$ and $2 \mid B$ since $2 \parallel C$. Similarly, we have $2 \mid S$ and $2 \mid A$. Moreover, we have $\Delta t =RS-TC \equiv 2 \pmod 4$, so $2 \nmid T$. We also have $AB-T^2=\Delta c$, so $2 \nmid c$.

It follows that the ternary form $f(x_1,x_2.x_3)=ax^2_{1}+bx^2_{2}+cx^2_{3}+2rx_{2}x_3+2sx_{1}x_3+2tx_1x_{2}$ is properly primitive, and the reciprocal $F(X_1,X_2,X_3)=AX^2_{1}+BX^2_{2}+CX^2_{3}+2RX_{2}X_3+2SX_{1}X_3+2TX_{1}X_2$ of $f$ is improperly primitive.
\end{proof}

%
The following proposition shows that the ternary quadratic form is unique in the equivalent class.

\begin{proposition}(\cite[Theorem 34]{Dickson1930})\label{Dick8}
If $C$ is a prime or double an odd prime, two primitive ternary forms having the same invariants $\Omega$ and $\Delta$ are equivalent if they properly present the same binary form of determinant $-4\Omega C$.
\end{proposition}


\section{Endomorphism rings of supersingular elliptic curves}
In this section, we study the correspondence between endomorphism rings of supersingular elliptic curves defined over $\mathbb{F}_{p^2}$ and ternary quadratic forms with discriminant $p$. We divide into two cases. Firstly, we study endomorphism rings of supersingular elliptic curves defined over $\mathbb{F}_{p}$. Secondly, we study endomorphism rings of $\mathbb{Z}[\sqrt{-cp}]$-oriented supersingular elliptic curves.
\subsection{Supersingular elliptic curves over $\mathbb{F}_p$}

Let $p>3$ be a prime. If the supersingular elliptic curve $E$ is defined over $\mathbb{F}_p$, then $\text{End}(E)$ contains an element $\alpha$ with minimal polynomial $x^2+p$ (see \cite{MR3451433}). Equivalently, the endomorphism ring $\text{End}(E)$ contains a subring $\mathbb{Z}[\sqrt{-p}]$. Ibukiyama \cite{MR683249} has given an explicit description of all maximal orders in $B_{p, \infty}$ containing a root of $x^2+p=0$.

Choose a prime $q$ such that
\begin{align}\label{e1}
q \equiv 3 \pmod 8, \quad \left( \frac{p}{q} \right) =-1.
\end{align}
Then $B_{p,\infty}$ can be written as $B_{p,\infty}=\mathbb{Q}+\mathbb{Q}\alpha+\mathbb{Q}\beta+\mathbb{Q}\alpha\beta$ where $\alpha^2=-p$, $\beta^2=-q$ and $\alpha\beta=-\beta\alpha$. Let $r$ be an integer such that $r^2+p \equiv 0 \pmod  q$. Put
$$\mathcal{O}(q,r)=\mathbb{Z} + \mathbb{Z}\frac{1+\beta}{2} + \mathbb{Z} \frac{\alpha(1+\beta)}{2} + \mathbb{Z}\frac{(r+\alpha)\beta}{q}.$$
When $p \equiv 3 \pmod  4$, we further choose an integer $r'$ such that $(r')^2+p \equiv 0 \pmod {4q}$ and put $$\mathcal{O}'(q,r')=\mathbb{Z} + \mathbb{Z}\frac{1+\alpha}{2} + \mathbb{Z} \beta + \mathbb{Z}\frac{(r'+\alpha)\beta}{2q}.$$

Both $\mathcal{O}(q,r)$ and $\mathcal{O}'(q,r')$ are independent of the choices of $r$ and $r'$ up to isomorphism respectively. Ibukiyama's results \cite{MR683249} show that both $\mathcal{O}(q,r)$ and $\mathcal{O}'(q,r')$ are maximal orders of $B_{p,\infty}$ and the endomorphism ring $\text{End}(E)$ is isomorphic to $\mathcal{O}(q,r)$ or $\mathcal{O'}(q,r')$ with suitable choice of $q$ if $j(E) \in \mathbb{F}_p$.

Let $j \in \mathbb{F}_p$ be a supersingular $j$-invariant and $E(j)$ be the corresponding supersingular elliptic curve. We have $\text{End}(E(j)) \cong \mathcal{O}'(q,r')$ if $\frac{1+\pi}{2} \in \text{End}(E(j))$ and $\text{End}(E(j)) \cong \mathcal{O}(q,r)$ if $\frac{1+\pi}{2} \notin \text{End}(E(j))$ for some $q$ satisfying (\ref{e1}). Especially, we have $\text{End}(E(0)) \cong \mathcal{O}(3,1)$.

\begin{proposition}\label{zt1}
Let $p \ge 5$ be a prime. Every maximal order $\mathcal{O}(q,r)$ can represent a primitive binary quadratic form with discriminant $-16p$. Moreover, this binary quadratic form is not in the principal genus. If $p \equiv 3 \pmod 4$, then every maximal order $\mathcal{O}'(q,r')$ can represent a binary quadratic form whose primitive form has discriminant $-p$.
\end{proposition}
\begin{proof}
In the proof of \cite[Theoreom 1]{MR1040429}, Kaneko has shown that every maximal order $\mathcal{O}(q,r)$ can represent a primitive binary quadratic form $(q,4r,\frac{4r^2+4p}{q})$ with discriminant $-16p$. It is easy to show that this form is not in the principal genus \cite{xiao2022endomorphism}. Similarly, every maximal order $\mathcal{O}'(q,r')$ can represent a binary quadratic form $(4q,4 r',\frac{(r')^2+p}{q})$, and its primitive form $(q, r',\frac{(r')^2+p}{4q})$ has discriminant $-p$.
\end{proof}

\begin{theorem}\label{zt2}
Let $\rho$ be a binary quadratic form with discriminant $-16p$. If $\rho$ is not in the principal genus or $\rho=4\rho'$ where $\rho'$ is primitive, then $\rho$ can be represented properly by a unique primitive ternary quadratic form $f$. Moreover, the reciprocal $F$ of $f$ is improperly primitive with discriminant $p$.
\end{theorem}

\begin{proof}
Let $\rho=(a,2t,b)$ be a primitive binary quadratic form with discriminant $-16p$. Moreover, we also assume that $\rho$ is not in the principal genus. Let $C=2$, $\Delta=1$ and $\Omega=2p$. It is easy to show that the congruences (\ref{equation2}) are solvable. By Proposition \ref{Dick5}, $\rho$ can be represented by a primitive ternary quadratic form $f$ with invariants $\Omega=2p$ and $\Delta=1$. Moreover, the ternary form $f$ is unique in the equivalent class by Proposition \ref{Dick8}. The reciprocal $F$ of $f$ is improperly primitive with discriminant $p$.

By Proposition \ref{Dick6}, similar results hold for $\rho=4\rho'$ where $\rho'$ is primitive.
\end{proof}
\begin{remark}
We can compute the maximal order in $B_{p,\infty}$ corresponding to the ternary quadratic form $F$ by the method in Section 2.3. Moreover, this maximal order is isomorphic to some $\mathcal{O}(q,r)$ or $\mathcal{O}'(q,r')$.
\end{remark}
In fact, Proposition \ref{zt1} and Theorem \ref{zt2} establish a correspondence between maximal orders in $B_{p,\infty}$ which contain a subring $\mathbb{Z}[\sqrt{-p}]$ and primitive binary quadratic forms with discriminant $-16p$ or $-p$.

\begin{theorem}\label{zt3}
Let $E$ be a supersingular elliptic curve over $\mathbb{F}_{p^2}$. Assume that the endomorphism ring $\operatorname{End}(E) \cong \mathcal{O}$ corresponds to a ternary quadratic form $F$ with discriminant $p$. Then $E$ is defined over $\mathbb{F}_p$ if and only if $F$ can represent $2$ properly.
\end{theorem}
\begin{proof}
By Proposition \ref{zt1}, the supersingular elliptic curve $E$ is defined over $\mathbb{F}_p$ if and only if $\operatorname{End}(E) \cong \mathcal{O}$ can represent a binary reduced quadratic form with discriminant $-16p$.

In the proof of Proposition \ref{Dick5}, the ternary quadratic form $F$ can represent $C$ properly, and we have $C=2$ in the proof of Theorem \ref{zt2}.
\end{proof}

\begin{theorem}\label{zt4}
Let $D$ be a prime with $D < p$ (resp. $4D<p$). Let $E$ be a supersingular elliptic curve defined over $\mathbb{F}_p$. If the imaginary quadratic order $O$ with discriminant $-D$ (resp. $-4D$) can be embedded into the endomorphism ring $\operatorname{End}(E)\cong \mathcal{O}$, one can compute $\mathcal{O}$ by solving one square root in $\mathbb{F}_D$.
\end{theorem}
\begin{proof}
Since the imaginary quadratic order $O$ can be embedded into the endomorphism ring of $E(j)$, we have that $j$ is a $\mathbb{F}_p$-root of the Hilbert class polynomial $H_{-D}(X)$ (or $H_{-4D}(X)$) and $\left( \frac{-D}{p} \right)=-1$. If $D \equiv 1 \pmod 4$, then $\left( \frac{D}{p} \right)=1$ by \cite[Theorem 1.1]{MR4432517}. It follows that $p \equiv 3 \pmod 4$ in this case. Moreover, the Hilbert class polynomial $H_{-4D}(X)$ has two $\mathbb{F}_p$-roots, and one can distinguish these two supersingular elliptic curves by $2$-torsion points. If $D \equiv 3 \pmod 4$, then the Hilbert class polynomial $H_{-D}$ (or $H_{-4D}(X)$) has only one $\mathbb{F}_p$-roots.

In all cases, we have $\left( \frac{-p}{D} \right)=1$, so the equation $x^2 \equiv -16p \pmod {16 D}$ (or $x^2 \equiv -16p \pmod {4 D}$) with $-4D < x \le 4D$ (or $-2D < x \le 2D$) is solvable. Moreover, the solution is unique up to $\pm 1$, so we get a binary form with discriminant $-16p$ or $-p$. By Theorem \ref{zt2}, one can compute a ternary quadratic form $F$ with discriminant $p$ easily. In general, In order to compute the endomorphism ring $\text{End}(E(j)) \cong \mathcal{O}$, we just need to solve one square root in $\mathbb{F}_D$.
\end{proof}

\begin{remark}
By Ibukiyama's theorem \cite{MR683249}, the prime $q$ satisfies $q \equiv 3 \pmod 8$. In Theorem \ref{zt4}, we do not need $D \equiv 3 \pmod 8$.

However, for $D >p$, there are more then one $\mathbb{F}_p$-roots of $H_{-D}(X)$ (or $H_{-4D}(X)$).  We can compute only one ternary quadratic form corresponding to the endomorphism ring of one curve.
\end{remark}

\begin{example}\label{ex1}
Let $p=83$. We have $\left( \frac{-7}{83} \right)=-1$. The $\mathbb{F}_p$-root of Hilbert class polynomial $H_{-7}(X)\equiv X-28 \pmod {83}$ is a supersingular j-invariant.

The elliptic curve $E:y^{2}=x^{3}+77x+12$ with $j(E)=28$ is a supersingular elliptic curve over $\mathbb{F}_{83}$. Solving the equation $t^2\equiv -16p \pmod {4 \times 7}$, we can compute a binary quadratic form $(7,4,48)$. Moreover, we can compute a ternary quadratic form $F(x,y,z)=2x^2+4y^2+24z^2-2yz-2xy$ with discriminant $83$, and it corresponds to a maximal order $\mathcal{O}=\mathbb{Z}+\mathbb{Z}i+\mathbb{Z}j+\mathbb{Z}k$, where $i,j,k$ satisfy the following equations:
$$ \begin{array}{ll}
  i^2=-i-24 \quad \quad & jk=\bar{i}=-1-i \vspace{0.5ex}  \\
  j^2=-12 \quad \quad & ki=2\bar{j}=-2j \vspace{0.5ex}  \\
  k^2=-k-2 \quad \quad & ij=12\bar{k}=12(-1-k).
 \end{array}
$$

By Ibukiyama's theorem \cite{MR683249}, the prime $59 \equiv 3 \pmod 8$ satisfies $\left( \frac{83}{59} \right)=-1$, so $\text{End}(E) \cong \mathcal{O}(59,r) \cong \mathcal{O}$.

\end{example}

\subsection{$\mathbb{Z}[\sqrt{-cp}]$-oriented supersingular elliptic curves}

Let $c\neq p$ be a prime. A supersingular elliptic curve $E$ is $\mathbb{Z}[\sqrt{-cp}]$-oriented or $\mathbb{Z}[\frac{1+\sqrt{-cp}}{2}]$-oriented if and only if there exists a $c$-isogeny $\phi:E \to E^p$, where $E^p$ is defined by $p$-powering all of the coefficients in the defining equation of $E$. Let $G \subset E[c]$ be the kernel of $\phi$. Then there exists a dual isogeny $\hat{\phi}:E^p \to E$ with $\ker(\hat{\phi})=\pi(G)=G^p$, where $\pi$ is the $p$-powering Frobenius map.

By \cite[Theorem 1]{Arpin2203}, the set $\text{End}(E,G)=\{\theta \in \text{End}(E):\theta(G) \subseteq G\}$ is isomorphic to an Eichler order in $B_{p,\infty}$ of level $|G|$. We review the properties of $\text{End}(E,G)$, and give a way to compute $\text{End}(E)$ from $\text{End}(E,G)$ in this subsection.

Choose a prime $q$ such that
\begin{equation}\label{equation31}
\left \{
\begin{array}{ll} \vspace{1ex}
  q \equiv 3 \pmod 8, \quad \left( \frac{p}{q} \right) =-1, \quad  \left( \frac{c}{q} \right) = 1, & \text{if}  \ c \ \text{is odd}; \\
  q \equiv 7 \pmod 8, \quad \left( \frac{p}{q} \right) =-1, & \text{if} \ c=2.
\end{array} \right.
\end{equation}
Let $r$ be an integer such that
\begin{equation}\label{equation32}
r^2+cp \equiv 0 \pmod q.
\end{equation}
Besides, when $cp \equiv 3 \pmod 4$, we choose an integer $r'$ such that
\begin{equation}\label{equation33}
(r')^2+cp \equiv 0 \pmod {4q}.
\end{equation}

\begin{proposition}(\cite[Theorem 1]{xiao2023oriented})\label{zt5}
There exist $\alpha', \beta \in B_{p,\infty}$ such that $(\alpha')^2=-cp$, $\beta^2=-q$ and $\alpha' \beta =-\beta  \alpha'$. Moreover, we have $B_{p,\infty } \cong \mathbb{Q}+\mathbb{Q}\alpha' +\mathbb{Q} \beta +\mathbb{Q}\alpha' \beta$.
Denote notions as $(\ref{equation31})$, $(\ref{equation32})$ and $(\ref{equation33})$, the lattice $\mathcal{O}_c(q,r) = \mathbb{Z}+\mathbb{Z}\frac{1+\beta}{2}+\mathbb{Z}\frac{\alpha'(1+\beta)}{2} +\mathbb{Z}\frac{(r+\alpha')\beta}{q}$ is an Eichler order in $B_{p, \infty}$ of level $c$. Moreover, if $cp \equiv 3 \pmod 4$, the lattice $\mathcal{O}'_c(q,r') = \mathbb{Z}+\mathbb{Z}\frac{1+\alpha'}{2}+\mathbb{Z}{\beta} +\mathbb{Z}\frac{(r'+\alpha')\beta}{2q}$ is also an Eichler order in $B_{p, \infty}$ of level $c$.
\end{proposition}

\begin{proposition}\cite[Theorem 4]{xiao2023oriented}
Let $c$ be a prime with $c < 3p/16$.
If $E$ is a $\mathbb{Z}[\sqrt{-cp}]$-oriented (resp. $\mathbb{Z}[\frac{1+\sqrt{-cp}}{2}]$-oriented) supersingular elliptic curve, then there exists a $c$-isogeny $\phi: E \to E^p$ with $\ker(\phi)=G$ and $\operatorname{End}(E,G)$ is isomorphic to some $\mathcal{O}_c(q,r)$ (resp. $\mathcal{O}'_c(q,r')$).
\end{proposition}

Let $\Lambda(q)=(\chi_1(q),\chi_2(q),\delta(q))$ (resp. $\Lambda'(q)=(\chi_1(q),\chi_2(q))$) be the assigned values of $q$ associated with $-16cp$ (resp. $-cp$), where $\chi_1(q)=\left( \frac{q}{p} \right)$, $\chi_2(q)=\left( \frac{q}{c} \right)$ and $\delta(q)=(-1)^{(q-1)/2}=-1$. The binary quadratic forms with associated values $\Lambda(q)$ (resp. $\Lambda'(q)$) are called in the genus class $\Lambda(q)$ (resp. $\Lambda'(q)$). Note that the binary quadratic form $\rho$ is in the genus class $\Lambda(q)$ (resp. $\Lambda'(q)$) if and only if $\rho^{-1}$ is also in the genus class $\Lambda(q)$ (resp. $\Lambda'(q)$).

\begin{proposition}\cite[Proposition 4.1]{xiao2023oriented}
Given an Eichler order $\mathcal{O}_c(q,r)$ (resp. $\mathcal{O}'_c(q,r')$), it can represent a reduced quadratic form $\rho$ or $\rho^{-1}$ in the genus class $\Lambda(q)$ (resp. $\Lambda'(q)$) respectively.
\end{proposition}
In the following, we show that the binary quadratic form $\rho$ or $\rho^{-1}$ can be represented by a ternary quadratic form, which corresponds to a maximal order in $B_{p,\infty}$.
\begin{theorem}\label{zt6}
Given an Eichler order $\mathcal{O}_c(q,r)$ or $\mathcal{O}'_c(q,r')$, one can compute a maximal order in $B_{p,\infty}$ which contains it by solving two square roots in $\mathbb{F}_c$.
\end{theorem}
\begin{proof}
Given an Eichler order $\mathcal{O}_c(q,r)$, we can assume that it can represent a binary quadratic form $\rho=(a, t',b)$ or $\rho^{-1}=(a, -t',b)$ with discriminant $-16cp$. We can write $t'=2t$ since $t'$ is even.

As we know, the forms $\rho$ and $\rho^{-1}$ are in the genus class $\Lambda(q)$. We have $(2t)^2-4ab=-16cp$. Let $C=2c$, $\Omega=2p$ and $\Delta=1$. We need to prove that the congruences in (\ref{equation2}) are solvable.

Since the forms $\rho$ and $\rho^{-1}$ are in the genus class $\Lambda(q)$, the equations $R^2 \equiv -a \pmod {C}$ and $S^2 \equiv -b \pmod{C}$ are solvable. We can assume that there exist $A, B \in \mathbb{Z}$ satisfying $BC-R^2=a$ and $AC-S^2=b$. Since $4 \mid a$ or $a \equiv 3 \pmod 4$, we have $4 \mid BC$. Similarly, $4 \mid AC$. We have
$$\begin{array}{ll}
  2pC&=ab-t^2 =(BC-R^2)(AC-S^2)-t^2  \\
   & =ABC^2-ACR^2-BCS^2+(RS)^2-t^2.
\end{array}$$

If $c$ is an odd prime, then $A$ and $B$ are even integers.
So $C \mid (RS)^2-t^2$. Moreover, we have $2 \mid (RS)^2-t^2$ which means $2 \mid RS\pm t$.

We have $c \mid RS-t$ or $c \mid RS+t$.
If only one of $RS-t \equiv 0 \pmod{2c}$ and $RS+t \equiv 0 \pmod{2c}$ holds, then the binary quadratic form $\rho$ or $\rho^{-1}$ can be represented properly by a ternary quadratic form $f$ by Proposition \ref{Dick5}. Moreover, the reciprocal $F$ of $f$ is improperly primitive with discriminant $p$. The form $F$ is unique in the equivalent class by Proposition \ref{Dick8}.

If $c \mid RS-t$ and $c \mid RS+t$, then $c \mid t$ and $c \mid RS$. It follows $c \mid ab$. We can assume $c \mid a$. Without loss of generality, we can assume $2\nmid a$, so we can choose $R=c$ and $T=(RS\pm t)/C=S/2 \pm t/C$. The two ternary quadratic forms $F_1(X_1,X_2,X_3)=AX^2_{1}+BX^2_{2}+CX^2_{3}+CX_2X_3+2SX_{1}X_3+(S+ 2t/C)X_{1}X_2$ and $F_2(X_1,X_2,X_3)=AX^2_{1}+BX^2_{2}+CX^2_{3}+CX_2X_3+2SX_{1}X_3+(S- 2t/C)X_{1}X_2$ are equivalent since there exists $U \in \operatorname{GL}_3(\mathbb{Z})$ with $\det(U)=\pm 1$ such that $A_{F_1} = U A_{F_2} U^{\rm t}$. Similar results hold for $2 \mid a$.

In the following, we assume $c=2$. We have $8 \mid (RS)^2-t^2$.
If $2 \nmid t$, then $2 \nmid ab$ and $2 \nmid RS$. Only one of $RS+t$ and $RS-t$ can be divided by $4$.

If $ 2 \mid t$, then $4 \mid ab$. Since $R$ satisfies $R^2+a \equiv 0 \pmod 4$, we can choose $R=0$ or $R=2$. If $R=0$, then $2T=\pm 2t /C =\pm t/2$. The two ternary forms $F_1(X_1,X_2,X_3)=AX^2_{1}+BX^2_{2}+4X^2_{3}+2SX_{1}X_3+t/2X_{1}X_2$ and $ F_2(X_1,X_2,X_3)=AX^2_{1}+BX^2_{2}+4X^2_{3}+2SX_{1}X_3-t/2X_{1}X_2$ are equivalent. If $R=2$, then $2T=2(RS\pm t)/C=S\pm t/2$. The ternary forms $F_1(X_1,X_2,X_3)=AX^2_{1}+BX^2_{2}+4X^2_{3}+4X_2X_3+2SX_{1}X_3+(S+ t/2)X_{1}X_2$ and $F_2(X_1,X_2,X_3)=AX^2_{1}+BX^2_{2}+4X^2_{3}+4X_2X_3+2SX_{1}X_3+(S-t/2)X_{1}X_2$ are equivalent.

%

In all cases, we get a unique ternary quadratic form in the equivalent class which can represent the form $\rho$ or $\rho^{-1}$. Moreover, its reciprocal form corresponds to a maximal order in $B_{p,\infty}$. Similar results hold for Eichler order $\mathcal{O}'_c(q,r')$.
\end{proof}

\begin{theorem}\label{zt7}
Let $E$ be a supersingular elliptic curve over $\mathbb{F}_{p^2}$ whose endomorphism ring $\operatorname{End}(E)$ corresponds to a ternary quadratic form $F$ with discriminant $p$. Let $c$ be a prime with $c < 3p/16$. The form $F$ can represent $2c$ properly if and only if $E$ is $\mathbb{Z}[\sqrt{-cp}]$-oriented or $\mathbb{Z}[\frac{1+\sqrt{-cp}}{2}]$-oriented.
\end{theorem}
\begin{proof}
If the form $F$ can represent $2c$ properly, then we can write $F(X_1,X_2,X_3)=AX_1^2+BX_2^2+2cX_3^2+2RX_2X_3+2SX_1X_3+2TX_1X_2$ by \cite[Theorem 10]{Dickson1930}. The reciprocal $f$ of $F$ can be written as $f(x_1,x_2.x_3)=ax^2_{1}+bx^2_{2}+c'x^2_{3}+2rx_{2}x_3+2sx_{1}x_3+2tx_1x_{2}$. The form $f$ can represent a binary quadratic form $\rho=(a, 2t,b)$. By the relation between $f$ and $F$, we have $ab-t^2=4cp$, so the form $\rho$ has discriminant $ -16cp$.

Assume that $\rho$ is primitive. The form $F$ corresponds to a maximal order $\mathcal{O}=\mathbb{Z}+\mathbb{Z}i+\mathbb{Z}j+\mathbb{Z}k$, where the $i$, $j$, $k$ satisfy the following equations.
\begin{align}
 \begin{array}{ll}
  i^2=Ri-Bc/2 \quad & jk=A\bar{i}/2=A(R-i)/2 \vspace{0.5ex}  \\
  j^2=Sj-Ac/2 \quad & ki=B\bar{j}/2=B(S-j)/2 \vspace{0.5ex} \\
  k^2=Tk-AB/4 \quad & ij=c\bar{k}=c(T-k).
 \end{array}
\end{align}

Using equations in $(\ref{equation3})$ with $\Delta=1$, the element $\alpha=(2i-R)(2j-S)+RS-2cT\in \mathcal{O}$ satisfies the equation $x^2+4cp=0$, so the imaginary quadratic order $\mathbb{Z}[\alpha]=\mathbb{Z}[\sqrt{-cp}]$ can be embedded into $\text{End}(E)$.

If the form $\rho$ is not primitive, then $\rho/4$ is primitive. Similarly, the imaginary quadratic order $\mathbb{Z}[\frac{1+\sqrt{-cp}}{2}]$ can be embedded into $\text{End}(E)$.

On the contrary, if $E$ is a $\mathbb{Z}[\sqrt{-cp}]$-oriented or $\mathbb{Z}[\frac{1+\sqrt{-cp}}{2}]$-oriented supersingular elliptic curve, then there exists a $c$-isogeny $\phi :E \to E^p$ with $\ker (\phi)=G$. We have that $\text{End}(E,G)$ is isomorphic to an Eichler $\mathcal{O}_c(q,r)$ or $\mathcal{O}'_c(q,r')$. As we know, the Eichler $\mathcal{O}_c(q,r)$ or $\mathcal{O}'_c(q,r')$ can represent a binary form with discriminant $-16cp$. By Theorem \ref{zt6}, the form $F$ can represent $C=2c$ properly.
\end{proof}
\begin{theorem}
Let $E$ be a supersingular elliptic curve over $\mathbb{F}_{p^2}$ with $j(E) \in \mathbb{F}_{p^2} \backslash \mathbb{F}_p$. Under GRH, the curve $E$ is $\mathbb{Z}[\sqrt{-cp}]$-oriented or $\mathbb{Z}[\frac{1+\sqrt{-cp}}{2}]$-oriented for some prime $c< p^{2/3} (\log p)^{2+\varepsilon}$.
\end{theorem}
\begin{proof}
For any supersingular elliptic curve $E$ over $\mathbb{F}_{p^2}$, there exists an element $\alpha \in \text{End}(E)\backslash \mathbb{Z}$ with $\text{nrd}(\alpha) < \frac{1}{2}p^{2/3}+\frac{1}{4}$ (see \cite[Section 4]{Elikes87}). So there exists an imaginary quadratic order $O$ which can be embedded into $\text{End}(E)$ with $|\text{disc}(O)| < 2p^{2/3}+1$.

Assume that $\text{End}(E)$ corresponds to a ternary quadratic form $F$ with discriminant $p$. Then the reciprocal $f$ of $F$ can represent an integer less than $2p^{2/3}+1$. Moreover, the ternary quadratic form $F$ can represent a binary quadratic form $\varrho$ properly with discriminant $|D| < 4(2p^{2/3}+1)$. Since the discriminant of $F$ is $p$, we have that $\varrho$ or $\varrho/2$ is primitive. Under GRH, every primitive binary quadratic form of discriminant $D$ can represent a prime smaller than $|D|\log (|D|)^{2+\varepsilon}$ \cite{sardari2018prime}.
So $\varrho$ can represent a prime less than $|D| (\log |D|)^{2+\varepsilon} < p^{2/3} (\log p)^{2+\varepsilon} $.

It follows that the ternary quadratic form $F$ can represent a prime less than $p^{2/3} (\log p)^{2+\varepsilon}$ properly, so the curve $E$ is $\mathbb{Z}[\sqrt{-cp}]$-oriented or $\mathbb{Z}[\frac{1+\sqrt{-cp}}{2}]$-oriented for some prime $c< p^{2/3} (\log p)^{2+\varepsilon}$ by Theorem \ref{zt7}.
\end{proof}

\begin{theorem}
Let $p>3$ and $c < 3p/16$ be primes. Let $E$ be a $\mathbb{Z}[\sqrt{-cp}]$-oriented (or $\mathbb{Z}[\frac{1+\sqrt{-cp}}{2}]$-oriented) supersingular elliptic curve defined over $\mathbb{F}_{p^2}$. If an imaginary quadratic order $O$ with discriminant $-4D$ (resp. $-D$ if $D \equiv 3 \pmod 4$) can be embedded into the endomorphism ring $\operatorname{End}(E) \cong \mathcal{O}$ where $D$ is a prime with $4D <p$ (resp. $D<p$), we can compute $\mathcal{O}$ by solving one square root in $\mathbb{F}_D$ and two square roots in $\mathbb{F}_c$.
\end{theorem}
\begin{proof}
Similar to the proof of Theorem \ref{zt4}. Firstly, we need to solve the equation $x^2\equiv -16cp \pmod {16D}$ (or $x^2\equiv -16cp \pmod {4D}$). This equation has a unique solution up to $\pm 1$, so we get a binary quadratic form $(a,2t,b)$ ($a=4D$ or $D$) with discriminant $(2t)^2-4ab=-16cp$.

By Theorem \ref{zt6}, the congruences $(\ref{equation2})$ are solvable, and we can compute a ternary quadratic form $F$ with discriminant $p$. It corresponds to a maximal order $\mathcal{O}$ in $B_{p,\infty}$, which is isomorphic to $\text{End}(E)$.
\end{proof}

\begin{example}\label{ex2}
Let $p=101$ and $c=3$. Define $\mathbb{F}_{p^2}=\mathbb{F}_p(\alpha)$ with $\alpha^2+2=0$. Since $\left( \frac{-11}{83} \right)=-1$, the root of Hilbert class polynomial $H_{-11}(X) \equiv X-57 \pmod {101}$ is a supersingular j-invariant.

The elliptic curve $E:y^{2}=x^{3}+39 x+23$ with $j(E)=57$ is a supersingular elliptic curve over $\mathbb{F}_{101}$. Solving the equation $t^2\equiv -16cp \pmod {4 \times 11}$, we can compute a binary quadratic forms $(11, 6,111)$. Moreover, we can compute a ternary quadratic form $F(x,y,z)=6x^2+2y^2+20z^2+2yz+6xz+2xy$ with discriminant $101$, and it corresponds to a maximal order $\mathcal{O}=\mathbb{Z}+\mathbb{Z}i+\mathbb{Z}j+\mathbb{Z}k$, where $i,j,k$ satisfy the following equations:
$$ \begin{array}{ll}
  i^2=i-10 \quad \quad & jk=3\bar{i}=3(1-i) \vspace{0.5ex}  \\
  j^2=3j-30 \quad \quad & ki=\bar{j}=3-j \vspace{0.5ex}  \\
  k^2=k-3 \quad \quad & ij=10\bar{k}=10(1-k).
 \end{array}
$$

\end{example}

\section{Isogenies between supersingular elliptic curves}
Let $\ell$ be a prime. Assume that $\varphi:E \to E'$ is an $\ell$-isogeny. If the endomorphism ring $\mathcal{O}$ of $E$ (which corresponds to a ternary quadratic form $F$) is known, we will compute the endomorphism ring $\mathcal{O}'$ of $E'$ and the corresponding ternary quadratic form in this section.

We divide into two cases. If the isogeny $\varphi$ is $\mathbb{Z}[\sqrt{-cp}]$-oriented, then it can be represented by an ideal in $\mathbb{Z}[\sqrt{-cp}]$. By \cite[Theorems 5 and 6]{xiao2023oriented}, we can compute an Eichler order in the endomorphism ring of $E'$. Moreover, we can compute the ternary quadratic form corresponding to $\text{End}(E')$ by the method in Section 4.
If the isogeny $\varphi$ is not $\mathbb{Z}[\sqrt{-cp}]$-oriented, then we can give the kernel ideal of $\varphi$, and represent its right order by a ternary quadratic form with discriminant $p$.

\subsection{$\mathbb{Z}[\sqrt{-cp}]$-oriented isogenies between supersingular elliptic curves}
Let $c$ be a prime with $c< 3p/16$ or $c=1$. Let $K=\mathbb{Q}(\sqrt{-cp})$ be an imaginary quadratic field. Let $O$ be an order in $K$ and $O_K$ be the integer ring of $K$. Let $E$ be a $\mathbb{Z}[\sqrt{-cp}]$-oriented (or $\mathbb{Z}[\frac{1+\sqrt{-cp}}{2}]$-oriented) supersingular elliptic curve. Notice that $E$ is defined over $\mathbb{F}_p$ if $c=1$.

In this section, we restrict to $\mathbb{Z}[\sqrt{-cp}]$-oriented isogenies with prime degree $\ell$. As we know, there are $1+\left( \frac{-cp}{\ell} \right)$ many $\mathbb{Z}[\sqrt{-cp}]$-oriented $\ell$-isogenies from $E$. If $\ell=c$, then $\left( \frac{-cp}{c} \right)=0$ and the $c$-isogeny $\phi:E \to E^p$ is oriented. In this case, we have $\text{End}(E)\cong \text{End}(E^p)$.

\begin{theorem}
Let $c<3p/16$ be a prime or $c=1$. Let $E$ be a $\mathbb{Z}[\sqrt{-cp}]$-oriented (or $\mathbb{Z}[\frac{1+\sqrt{-cp}}{2}]$-oriented) supersingular elliptic curve over $\mathbb{F}_{p^2}$. There exists a $c$-isogeny $\phi:E \to E^p$ with kernel $G \subseteq E[c]$. Assume that the endomorphism ring $\operatorname{End}(E,G)$ is known. If $\varphi:E \to E'$ is a $\mathbb{Z}[\sqrt{-cp}]$-oriented (or $\mathbb{Z}[\frac{1+\sqrt{-cp}}{2}]$-oriented) isogeny with prime degree $\ell \neq c$, then we can compute $\operatorname{End}(E')$ by composing two binary quadratic forms with discriminant $-16cp$ and solving two square roots in $\mathbb{F}_c$.
\end{theorem}
\begin{proof}
By \cite[Theorem 4]{xiao2023oriented}, the endomorphism ring $\text{End}(E,G)$ is isomorphic to an Eichler order $\mathcal{O}_c(q,r)$ (or $\mathcal{O}'_c(q,r')$) for some $q$ satisfying (\ref{equation31}). Assume that the Eichler order $\mathcal{O}_c(q,r)$ (or $\mathcal{O}'_c(q,r')$) can represent a binary quadratic form $\rho$ with discriminant $-16cp$.


The $\mathbb{Z}[\sqrt{-cp}]$-oriented (or $\mathbb{Z}[\frac{1+\sqrt{-cp}}{2}]$-oriented) $\ell$-isogeny $\varphi:E \to E'$ can be represented by a binary quadratic form $\varrho$ with discriminant $-16cp$. By \cite[Theorem 5 and 6]{xiao2023oriented}, $\text{End}(E',\varphi(G))$ can represent a quadratic form $\rho'=\rho \varrho^2$ with discriminant $-16cp$. Moreover, we can compute a ternary quadratic form $f'$ which can represent $\rho'$, and the reciprocal $F'$ of $f'$ corresponds to $\text{End}(E')$.

\end{proof}

\begin{remark}
If $\ell=c$, then the binary quadratic form $\varrho$ has order $\le 2$. So $\rho'=\rho \varrho^2=\rho $, which implies $\text{End}(E') \cong \text{End}(E)$.
\end{remark}

\begin{example}\label{ex3}
As in Example \ref{ex1}, the elliptic curve $E_1:y^{2}=x^{3}+77 x+12$ over $\mathbb{F}_{83}$ is supersingular. The endomorphism ring $\text{End}(E_1) \cong \mathcal{O}_1$ can represent a binary quadratic form $(7,4,48)$. There exists a $\mathbb{Z}[\sqrt{-p}]$-oriented $3$-isogeny $\varphi:E_1 \to E_2$ with $E_2:y^2=x^3+18x+16$, and the kernel ideal of $\varphi$ corresponds to a binary quadratic form $(3,2,111)$.

The endomorphism ring $\text{End}(E_2) \cong \mathcal{O}_2$ can represent a binary quadratic form $(16,-12,23)=(7,4,48)(3,2,111)^2$. Moreover, we can compute a ternary quadratic form $F(x,y,z)=2x^2+8y^2+12z^2-6yz-2xz$ with discriminant $83$, and it corresponds to a maximal order $\mathcal{O}=\mathbb{Z}+\mathbb{Z}i+\mathbb{Z}j+\mathbb{Z}k$, where $i,j,k$ satisfy the following equations:
$$ \begin{array}{ll}
  i^2=-3i-24 \quad \quad & jk=\bar{i}=-3-i \vspace{0.5ex}  \\
  j^2=-j-6 \quad \quad & ki=4\bar{j}=4(-1-j) \vspace{0.5ex}  \\
  k^2=-4 \quad \quad & ij=6\bar{k}=-6k.
 \end{array}
$$
\end{example}

\subsection{Other isogenies between supersingular elliptic curves}

We first review the relations between binary quadratic forms with discriminants $D$ and $\ell^2 D$, and we study the action of non-oriented isogenies in the following.

In general, any primitive binary quadratic form with discriminant $D\ell^2$ can be derived from a primitive binary quadratic form of discriminant $D$ by the application of a transformation $L \in M_{2 \times 2}(\mathbb{Z})$ of determinant $\ell$ (see \cite[Chapter 7.1]{MR1012948}). We define two transformations $L_1$ and $L_2$ of discriminant $\ell$ to be $\operatorname{SL}_2(\mathbb{Z})$-right-equivalent if there exists a matrix $M \in \operatorname{SL}_2(\mathbb{Z})$ such that $L_1M=L_2$. Moreover, we have the following proposition.
\begin{proposition}(\cite[Proposition 7.2]{MR1012948})\label{p51}
Let $\ell$ be a prime integer. The $\operatorname{SL}_2(\mathbb{Z})$-right-equivalent transformations have as equivalence class representatives the $\ell+1$ transformations
\begin{center}
$L_{\ell}=\left(\begin{array}{cc}
1 \ & 0 \\
0 \ & \ell
\end{array}\right)$
\ and  \
$L_h=\left(\begin{array}{cc}
\ell \ & h \\
0  \ & 1
\end{array}\right),$
\end{center}
for $0 \le h \le \ell-1$.
\end{proposition}

\begin{proposition}(\cite[Proposition 7.3]{MR1012948})
Given any binary quadratic form $(a,t,b)$ with discriminant $D$, and any prime $\ell$ with $\ell \nmid a$, the $\ell +1$ representative transformations of determinant $\ell$ produce exactly $\ell -\left(\frac{D}{\ell} \right)$ primitive binary quadratic forms of determinant $D\ell^2$.
\end{proposition}

\begin{remark}\label{re1}
Let $\rho =(a, t, b)$ be a primitive binary quadratic form with discriminant $-16cp$ or $-cp$.
The form $\rho_h=(a\ell^2,\ell(t+2ah),ah^2+th+b)$ is derived from $\rho$ by applying the transformation $L_{h}$ for $0 \le h \le \ell-1$. The form $\rho_{\ell}=(a,t\ell,b\ell^2)$ is derived from $\rho$ by applying the transformation $L_{\ell}$.
\end{remark}

\begin{proposition}\label{Buell1}
Notations as remark $\ref{re1}$. Assume $\ell \nmid aD$, there are $\left( \ell -\left(\frac{D}{\ell} \right) \right)/2$ primitive binary quadratic forms which are derived from $\rho$ satisfying $\left(\frac{-\rho_h}{\ell}\right)=1$.
\end{proposition}
\begin{proof}
For $0 \le h \le \ell-1$, the form $\rho_h$ can represent $\rho(h,1)=ah^2+th+b$. The form $\rho_{\ell}$ can represent $a$. The form $\rho_h$ is primitive if and only if $\ell \nmid \rho(h,1)$ for $0 \le h \le \ell-1$.
 By \cite[Theorem 5.48]{Lidl_Niederreiter_1996}, we have
$$\left( \frac{a}{\ell} \right)+ \sum_{h=0}^{\ell-1} \left( \frac{\rho(h,1)}{\ell} \right) =\left( \frac{\rho_{\ell}}{\ell} \right)+ \sum_{h=0}^{\ell-1} \left( \frac{\rho_h}{\ell} \right) =0.$$
So there are $\left( \ell -\left( \frac{D}{\ell}\right)\right)/2$ forms derived from $\rho$ satisfying $\left(\frac{-\rho_h}{\ell}\right)=1$.
\end{proof}

Let $E$ be a $\mathbb{Z}[\sqrt{-cp}]$-oriented supersingular elliptic curve defined over $\mathbb{F}_{p^2}$. There exists a $c$-isogeny $\phi:E \to E^p$ with $\ker (\phi)=G$. Let $\ell$ be a prime. Let $\varphi:E \to E'$ be an $\ell$-isogeny with $\ker(\varphi)=H \neq G$. We assume that $\varphi$ is not $\mathbb{Z}[\sqrt{-cp}]$-oriented. The isogeny
$$\Psi: E' \xrightarrow{\hat{\varphi}} E \xrightarrow{\phi} E^p \xrightarrow{\varphi^p} (E')^p    $$
has degree $\ell^2 c$, where the kernel of $\varphi^p:E^p \rightarrow  (E')^p$ is $\pi(H)=H^p$. Since $\varphi$ is not oriented, we know that the kernel of $\Psi$ is a cyclic group. Denote it by $\ker(\Psi)=G'$. The endomorphism ring $\text{End}(E',G')$ is isomorphic to an Eichler order with level $\ell^2 c$ (see \cite[Theorem 3.7]{Arpin2203}).

We have $\Psi(\varphi(\ker\phi))=(\varphi^p \circ \phi \circ \hat{\varphi}) (\varphi(\ker \phi))=(\varphi^p \circ \phi)([\ell] \ker \phi)=\infty$. It follows that $\varphi(\ker \phi) \subseteq \ker \Psi=G'$ is a subgroup of order $c$.

If $\omega \in \text{End}(E',G')$, then $\omega (\varphi(\ker \phi))\subseteq G'$. Moreover, if $\omega (\varphi(\ker \phi))= \infty$, then $\omega (\varphi(\ker \phi)) \subseteq \varphi(\ker \phi)$ obviously. If $\omega (\varphi(\ker \phi))\neq \infty$, then $\omega (\varphi(\ker \phi))$ is a cyclic subgroup of $G'$ with order $c$. Since the cyclic subgroup of $G'$ with order $c$ is unique, we must have $\omega (\varphi(\ker \phi)) \subseteq \varphi(\ker \phi)$.

Moreover, we have
$$\begin{array}{rl}
    ([\ell]^{-1}\circ \hat{\varphi} \circ \omega \circ \varphi)(\ker \phi)
&=([\ell]^{-1}\circ \hat{\varphi})(\omega(\varphi(\ker \phi))) \\
&\subseteq ([\ell]^{-1}\circ\hat{\varphi})(\varphi(\ker \phi)) \\
&\subseteq \ker \phi=G.
  \end{array}
$$
It follows $[\ell]^{-1}\circ \hat{\varphi} \circ \omega \circ \varphi \in \text{End}(E,G)$, so we get an embedding
$$\begin{array}{rcl}
    \iota: \text{End}(E',G')& \hookrightarrow & \text{End}(E,G) \\
    \omega & \mapsto & [\ell]^{-1}\circ \hat{\varphi} \circ \omega \circ \varphi.
  \end{array}
$$

As we know, the endomorphism ring $\text{End}(E,G)$ is isomorphic to an Eichler order $\mathcal{O}_c(q,r)$ or $\mathcal{O}'_c(q,r')$ for some prime $q$ satisfying $(\ref{equation31})$, so $\text{End}(E',G')$ is isomorphic to a suborder of $\text{End}(E,G)$ with discriminant $\ell^4c^2p^2$. We begin to discuss the subgroups of $E[\ell]$.

As we know, there are $\ell+1$ subgroups of $E[\ell]$ with order $\ell$, among which the $1+\left( \frac{-cp}{\ell} \right)$ subgroups are the kernels of the $\mathbb{Z}[\sqrt{-cp}]$-oriented $\ell$-isogenies from $E$. Any other $\ell - \left( \frac{-cp}{\ell} \right) $ subgroups of $E[\ell]$ is the kernel of a non $\mathbb{Z}[\sqrt{-cp}]$-oriented $\ell$-isogeny from $E$.

Let $\beta,\alpha' \in B_{p,\infty}$ satisfy $\beta^2=-q$, $(\alpha')^2=-cp$ and $\beta \alpha'=-\alpha' \beta$. If $\left( \frac{-q}{\ell} \right)=1$, then the equation $x^2+q \equiv 0 \pmod {\ell}$ is solvable. Assume that $-\ell < \pm m < \ell$ are the solutions. There exists a point $P \in E[\ell]$ such that $[\beta] P=mP$, and $[\beta]([\alpha'] P)=-[\alpha']([\beta] P)=-[\alpha'](m P)=-m([\alpha'] P)$. We define that $H_{+}=\langle P \rangle$ and $H_{-}=\langle [\alpha']P \rangle$ are eigenvectors of $\beta$ in $E[\ell]$.




\begin{proposition}\label{p52}
If the binary quadratic form $\rho=(q,4r,\frac{4r^2+4cp}{q})$ in the genus $\Lambda(q)$ can represent two primes $q_1$ and $q_2$ satisfying
\begin{equation}\label{equation41}
\left \{
\begin{array}{ll} \vspace{1ex}
  q \equiv 3 \pmod 8, \quad \left( \frac{p}{q} \right) =-1, \quad  \left( \frac{c}{q} \right) = 1, \quad  \left( \frac{\ell}{q} \right) = 1, & \text{if}  \ c \ \text{is odd}; \\
  q \equiv 7 \pmod 8, \quad \left( \frac{p}{q} \right) =-1, \quad  \left( \frac{\ell}{q} \right) = 1, & \text{if} \ c=2,
\end{array} \right.
\end{equation}
then there exist two elements $\beta_1, \beta_2 \in \mathcal{O}_c(q,r)$ with $\operatorname{nrd}(\beta_1)=q_1$ and $\operatorname{nrd}(\beta_2)=q_2$. Moreover, $\beta_1$ and $\beta_2$ have the same eigenvectors in $E[\ell]$ if and only if $q_1q_2$ can be represented by the binary quadratic form $x^2+4 \ell^2cpy^2$.
\end{proposition}
\begin{proof}
If $q_1=\rho(x_1,z_1)$ and $q_2=\rho(x_2,z_2)$, then $\beta_1=x_1 \beta+z_1 \frac{2(r+\alpha')\beta}{q}$ and $\beta_1=x_2 \beta+z_2 \frac{2(r+\alpha')\beta}{q}$ with $\text{nrd}(\beta_1)=q_1$ and $\text{nrd}(\beta_2)=q_2$. By \cite[Corollary 3.1]{xiao2023oriented}, we have $\mathcal{O}_c(q,r) \cong \mathcal{O}_c(q_1,r_1) \cong \mathcal{O}_c(q_2,r_2)$.

We claim that $\beta_1$ and $\beta_2$ have the same eigenvectors in $E[\ell]$ if and only if there exists a $k \in \mathbb{F}_{\ell}^{*}$ such that $k(x_1,z_1)\equiv (x_2,z_2) \pmod {\ell}$.

If there exists a $k \in \mathbb{F}_{\ell}^{*}$ such that $k(x_1,z_1)\equiv (x_2,z_2) \pmod {\ell}$, then $\beta_1$ and $\beta_2$ have the same eigenvectors in $E[\ell]$ obviously. On the contrary, if $\beta_1$ and $\beta_2$ have the same eigenvectors $H_{+}=\langle P \rangle$ and $H_{-}=\langle P' \rangle$ in $E[\ell]$, then $E[\ell]=\{k_1 P+k_2 P':0\le k_1,k_2 < \ell \}$. Assume $[\beta_i] P= m_i P$ and $[\beta_i] P'=-m_i P'$, we have $(m_2 [\beta_1] -m_1 [\beta_2])P =(m_2 [\beta_1] -m_1 [\beta_2])P'= \infty$. It follows $E[\ell] \subseteq \ker(m_2 [\beta_1] -m_1 [\beta_2])$, which means that $\ell \mid m_2 [\beta_1] -m_1 [\beta_2]$. So we have $k(x_1,z_1)\equiv (x_2,z_2) \pmod {\ell}$ for $k=m_2/m_1$. This proves the claim.
%
%

Since $q_1$ can be represented by a binary quadratic form with discriminant $-16\ell^2cp$ which is derived from $\rho$, we can assume $L_{h_1} (s_1,t_1)^{\textrm{t}}=(x_1,z_1)^{\textrm{t}}$ where $L_{h_1}$ is defined in Proposition \ref{p51}. Similarly, we can assume $L_{h_2} (s_2,t_2)^{\textrm{t}}=(x_2,z_2)^{\textrm{t}}$. Then $k(x_1,z_1)\equiv (x_2,z_2) \pmod {\ell}$ implies $h_1=h_2$. It follows that $q_1$ and $q_2$ can be represented by the same binary quadratic form with discriminant $-16\ell^2cp$ which is derived from $\rho$. Equivalently, $q_1q_2$ can be represented by the binary quadratic form $x^2+4 \ell^2cpy^2$.
\end{proof}
\begin{remark}
If the endomorphism ring $\text{End}(E,G)$ is isomorphic to an Eichler order $\mathcal{O}'_c(q,r')$, then $\beta_1$ and $\beta_2$ have the same eigenvectors in $E[\ell]$ if and only if $4q_1q_2$ can be represented by the binary quadratic form $x^2+ \ell^2cpy^2$.
\end{remark}

Let $\beta_1, \beta_2 \in \mathcal{O}_c(q,r)$ satisfy conditions in Proposition \ref{p52}.
We define $\beta_1 \sim \beta_2$ if and only if they have the same eigenvectors.
Define $\Gamma=\{ \beta \in \mathcal{O}_c(q,r) : \beta^2=-q,\alpha'\beta=-\beta\alpha' \}/ \sim$, where the prime $q$ satisfies $(\ref{equation41})$.

If $\ell \nmid cp$, then there are $\left( \ell - \left( \frac{-cp}{\ell} \right) \right)/2$ primitive binary quadratic forms with discriminant $-16\ell^2cp$ derived from $\rho$, which can represent a prime $q$ satisfying $\left( \frac{-q}{\ell} \right)=1$ by Proposition \ref{Buell1}. It follows that there are $\left( \ell - \left( \frac{-cp}{\ell} \right) \right)/2$ elements in the set $\Gamma$, and each element has two eigenvectors in $E[\ell]$.
If $\ell=c$, then there are $ \ell $ elements in the set $\Gamma$. We have the following proposition.

\begin{proposition}
If $\ell=c$, then $\left( \frac{\ell}{q}\right)=1$ for every prime $q$ satisfying $(\ref{equation31})$. There exists only one oriented $\ell$-isogeny $\phi:E\to E^p$ with kernel $G$. Moreover, the group $G$ is an eigenvector of any $\beta \in \Gamma$.
\end{proposition}
\begin{proof}
We denote $E[\ell]=\langle P_1, P_2 \rangle$ and $G=\langle P_1 \rangle$. If $H=\langle P \rangle$ is an eigenvector of $\beta\in \Gamma$ with $[\beta] P=mP$, then $\langle [\alpha'] P \rangle$ is also an eigenvector of $\beta$ with $[\beta] ([\alpha'] P)=-m([\alpha'] P)$.

Assume $H=\langle P \rangle \neq G$, and we claim $\langle [\alpha'] P \rangle =G$. Write $P=P_2+kP_1$ with $k \in \{0,\ldots,\ell-1 \}$. Since $[\alpha']=\pi_p \circ \phi :E\to E^p \to E$, we have $\ker \phi \subseteq \ker [\alpha']$. So $[\alpha'] P=[\alpha'](P_2+kP_1)=[\alpha'] P_2$. Moreover, we have $\ker(\phi \circ [\alpha'])=\ker(\phi \circ \pi_p \circ \phi)=\ker(\pi_p \circ \phi^p \circ \phi) \supseteq E [\ell]$ since $\phi^p=\pm \hat{\phi}$. So we have $\phi([\alpha'] P)=(\phi \circ [\alpha']) P=\infty$, which implies $\langle [\alpha'] P\rangle=\langle [\alpha'] P_2 \rangle=\ker(\phi)=G$.
\end{proof}

\begin{proposition}\label{zt10}
If $E$ is a $\mathbb{Z}[\sqrt{-cp}]$-oriented (resp. $\mathbb{Z}[\frac{1+\sqrt{-cp}}{2}]$-oriented) supersingular elliptic curve over $\mathbb{F}_{p^2}$, then there exists a $c$-isogeny $\phi:E \to E^p$ with kernel $G$. Assume that the $\ell$-isogeny $\varphi:E \to E'$ is not $\mathbb{Z}[\sqrt{-cp}]$-oriented. Denote the kernel of
$\Psi: E' \xrightarrow{\hat{\varphi}} E \xrightarrow{\phi} E^p \xrightarrow{\varphi^p} (E')^p$
by $G'$, then $G'$ is a cyclic subgroup of $E'[\ell^2 c]$ with order $\ell^2 c$. Moreover, $\operatorname{End}(E',G')$ is isomorphic to an Eichler order $\mathcal{O}_{c\ell^2}(q,r)=\mathbb{Z} + \mathbb{Z}\frac{1+\beta}{2} + \mathbb{Z} \frac{\ell \alpha'(1+\beta)}{2} + \mathbb{Z}\frac{\ell( r+\alpha')\beta}{q}$ (resp. $\mathcal{O}'_{c\ell^2}(q,r')=\mathbb{Z}+\mathbb{Z}\frac{\ell(1+\alpha')}{2}+\mathbb{Z}\beta
+\mathbb{Z}\frac{\ell(r'+\alpha')\beta}{2q}$)  of level $\ell^2 c$ for some $\beta \in \Gamma$.
\end{proposition}
\begin{proof}
Let $G$ be the kernel of $\phi:E \to E^p$. Denote the kernel of isogeny $\Psi: E' \xrightarrow{\hat{\varphi}} E \xrightarrow{\phi} E^p \xrightarrow{\varphi^p} (E')^p$ by $G'$. Since the $\ell$-isogeny $\varphi:E \to E'$ is not $\mathbb{Z}[\sqrt{-cp}]$-oriented, the kernel $\ker(\varphi) \subseteq E[\ell]$ is an eigenvector of some $\beta \in \Gamma$.

If $\text{nrd}(\beta)=q$ and $m^2+q \equiv 0 \pmod {\ell}$, then $\ker(\varphi)=\{ P \in E[\ell] : [\beta]P=mP\}$. Assume $\text{End}(E,G)\cong \mathcal{O} \cong \mathbb{Z}+\mathbb{Z}\frac{1+\beta}{2}+\mathbb{Z}\frac{\alpha'(1+\beta)}{2} +\mathbb{Z}\frac{(r+\alpha')\beta}{q}$. The kernel ideal of $\varphi$ is $\mathcal{O}(\ell, \beta-m)$.

It is easy to show that $\frac{1+\beta}{2} \in \text{End}(E',G')$. We also have $\frac{\ell \alpha'(1+\beta)}{2}, \frac{\ell (r+\alpha')\beta}{q} \in \text{End}(E',G')$. So the order
$$\mathcal{O}_{c\ell^2}(q,r)=\mathbb{Z} + \mathbb{Z}\frac{1+\beta}{2} + \mathbb{Z} \frac{\ell \alpha'(1+\beta)}{2} + \mathbb{Z}\frac{\ell (r+ \alpha')\beta}{q}$$
is a suborder of $\text{End}(E',G')$. Moreover, the discriminant of order $\text{End}(E',G')$ is $\ell^4 c^2 p^2$, which implies $\text{End}(E',G') \cong \mathcal{O}_{c\ell^2}(q,r)$.

Similar results hold for $\mathbb{Z}[\frac{1+\sqrt{-cp}}{2}]$-oriented supersingular elliptic curves.
\end{proof}

\begin{remark}
The kernel ideal of isogeny $\varphi$ is $I=\mathcal{O}(\ell,\beta-m)$. One can compute the right order of $I$ by other ways, and we aim to represent the right order by a ternary quadratic form in this section.
\end{remark}

\begin{proposition}
Let $E$ be a $\mathbb{Z}[\sqrt{-cp}]$-oriented (or $\mathbb{Z}[\frac{1+\sqrt{-cp}}{2}]$-oriented) supersingular elliptic curve. Assume that $\langle P \rangle$ is an eigenvector of $\beta \in \Gamma$. There exist two $\ell$-isogenies $\varphi:E \to E'$ and $\psi:E \to E''$ with $\ker(\varphi)=\langle P \rangle$ and $\ker(\psi)=\langle \alpha' P \rangle$. Moreover, there exists a $c$-isogeny $\vartheta: (E')^p \to E''$.
\end{proposition}
\begin{proof}
If $\langle P \rangle$ is an eigenvector of $\beta$, then $\langle \alpha' P \rangle$ is the other eigenvector of $\beta$. So there exist two $\ell$-isogenies $\varphi:E \to E'$ with $\ker(\varphi)=\langle P \rangle$ and $\psi:E \to E''$ with $\ker(\psi)=\langle [\alpha'] P \rangle$.

Considering the isogeny $\psi \circ [\alpha']: E \to E \to E''$, we have $\psi \circ [\alpha'] (P) =\psi([\alpha'] P)=\infty$, which implies $\langle P \rangle =\ker(\varphi) \subseteq \ker (\psi \circ [\alpha'])$. Since the isogeny $\varphi$ is separable, there exists a unique isogeny $\varpi: E' \to E''$ with degree $cp$. Moreover, the isogeny $\varpi$ can be split into $E' \to (E')^p \to E''$, where the isogeny $\vartheta:(E')^p \to E''$ is of degree $c$.
\end{proof}
\begin{remark}
If $c=1$, then $(E')^p$ is isomorphic to $E''$, which implies that $\text{End}(E') \cong \text{End}(E'')$. If $\ell =c$, then the curve $E'$ (or $E''$) is isomorphic to $E^p$. Without loss of generality, we can assume $E' \cong E^p$, so the isogeny $\vartheta$ is isomorphic to $\psi$.
\end{remark}
The following theorem gives the endomorphism rings of $E'$ and $E''$.
\begin{theorem}
Let $E$ be a $\mathbb{Z}[\sqrt{-cp}]$-oriented (or $\mathbb{Z}[\frac{1+\sqrt{-cp}}{2}]$-oriented) supersingular elliptic curve over $\mathbb{F}_{p^2}$. There exists a $c$-isogeny $\phi:E \to E^p$ with kernel $G$. Assume that $\langle P \rangle \subseteq E[\ell]$ and $\langle \alpha' P \rangle \subseteq E[\ell]$ are the eigenvectors of $\beta \in \Gamma$, and $\langle P \rangle$ (resp. $\langle \alpha' P \rangle$) is the kernel of $\varphi:E \to E'$ (resp. $\psi:E \to E''$). Given the Eichler order in $B_{p, \infty}$ corresponding to $\operatorname{End}(E,G)$, one can compute (at most) two maximal orders in $B_{p, \infty}$ corresponding to $\operatorname{End}(E')$ and $\operatorname{End}(E'')$ by solving two square roots in the ring $\mathbb{Z}/ 2c\ell^2 \mathbb{Z}$.
\end{theorem}
\begin{proof}

By Proposition \ref{zt10}, the endomorphism ring $\text{End}(E,G) \cong \mathcal{O}_c(q,r)$ (resp. $\text{End}(E',G') \cong \mathcal{O}_{c\ell^2}(q,r)$) can represent a binary quadratic form $\rho$ (resp. $\rho'$) with discriminant $-16cp$ (resp. $-16c\ell^2 p$). Moreover, the form $\rho'=(a',2t',b')$ can be derived from $\rho$, and $\rho'$ can represent a prime $q$ satisfying $(\ref{equation41})$.

We first assume $\ell \nmid cp$. Without loss of generality, we can assume $\ell \nmid a'$. As we know, the form $(a',2t'+2ka',\star)$ with discriminant $-16\ell^2 cp$ is equivalent to $\rho'$ for some $\ell \nmid k$. We can choose a suitable $k \in \mathbb{Z}$ such that $c \mid (2t'+2ka')$ and $\ell \nmid (2t'+2ka')$. So we can assume $\rho'=(a',2t',b')$ with $\ell \nmid 2a'b't'$ and $c \mid t'$ in this case.

As we know, the form $\rho'$ can be represented by a ternary quadratic form if and only if the congruences in (\ref{equation2}) are solvable. Let $\Delta=1$ and $C'=2c\ell^2$. The equations $(R')^2+a' \equiv 0 \pmod {2c\ell^2}$ and $(S')^2+b' \equiv 0 \pmod {2c\ell^2}$ are solvable since $\left(\frac{-a'}{\ell} \right)=\left(\frac{-b'}{\ell} \right)=\left(\frac{-q'}{\ell} \right)=1$.
Let $B'C'-(R')^2=a'$ and $A'C'-(S')^2=b'$. We also have
$$\begin{array}{ll}
  4pc\ell^2&=a'b'-(t')^2 =(B'C'-(R')^2)(A'C'-(S')^2)-(t')^2  \\
   & =A'B'(C')^2-A'C'(R')^2-B'C'(S')^2+(R'S')^2-(t')^2.
\end{array}$$
So $c \ell^2 \mid (R'S')^2-(t')^2$. Since $c \mid t'$, we have $c \mid R'S'$. Moreover, we claim that either $\ell^2 \mid (R'S'-t')$ or $\ell^2 \mid (R'S'-t')$. If $\ell \parallel (R'S'-t')$ and $\ell \parallel (R'S'+t')$, then $\ell \mid t'$, which is a contradiction.
So we have $c\ell^2 \mid (R'S'-t')$ or $c\ell^2 \mid (R'S'-t')$. It follows that $R'$ and $S'$ satisfy the equation $R'S'\pm t' \equiv 0 \pmod {c\ell^2}$. If $c = 1$, then there exists only one choice of $R'$ and $S'$ in $(\mathbb{Z}/2c\ell^2\mathbb{Z})/\{\pm 1\}$. If $c$ is a prime, then there exist two choices of $R'$ and $S'$ in $(\mathbb{Z}/2c\ell^2\mathbb{Z})/\{\pm 1\}$.

As we see, the form $\rho'$ can represented by two (resp. one) reduced ternary quadratic forms if $c$ is a prime (resp. $c=1$). Their reciprocal forms are ternary quadratic forms with discriminant $p$, which correspond to maximal orders in $B_{p,\infty}$.

If $\ell=c$, then we can compute only one ternary quadratic form with discriminant $p$, which corresponds to the endomorphism ring of $E'$.
\end{proof}
\begin{remark}
If $c=1$ or $\ell=c$, then we can compute the endomorphism ring of $E'$ directly. If $\ell \neq c$ is a prime, then we need to distinguish $\text{End}(E')$ and $\text{End}(E'')$ by other ways.
\end{remark}

\begin{example}
As in Example \ref{ex3}, the elliptic curve $E_1:y^{2}=x^{3}+77 x+12$ over $\mathbb{F}_{83}$ is supersingular. Define $\mathbb{F}_{p^2}=\mathbb{F}_p(\alpha)$ with $\alpha^2+1=0$. There exist two $3$-isogenies $\psi_1:E_1 \to E_3$ and $\psi_2:E_1 \to E_4$ with $E_3:y^2=x^3+(15\alpha+52)x+(69\alpha+24)$ and $E_4:y^2=x^3+(68\alpha+52)x+(14\alpha+24)$ which are not $\mathbb{Z}[\sqrt{-cp}]$-oriented. Evidently, $E_4=E_3^p$ and $\text{End}(E_3) \cong \text{End}(E_4)$.

As we know, $\text{End}(E_1)$ can represent a binary quadratic form $\rho=(7,4,48)$. The binary quadratic form $(59,64,68)$ is the only form derived from $\rho$ and $\left( \frac{3}{59} \right)=1$ with discriminant $-16 \times 3^2 \times p$. Moreover, the form $(59,64,68)$ can be represented by a ternary quadratic form $f(x,y,z)=18x^2+6y^2+4z^2-2yz+4xz+14xy,$
which is equivalent to
$$f(x,y,z)=4x^2+6y^2+8z^2+4yz-2xz-2xy.$$
It corresponds to a maximal order $\mathcal{O}=\mathbb{Z}+\mathbb{Z}i+\mathbb{Z}j+\mathbb{Z}k$, where $i,j,k$ satisfy the following equations:
$$ \begin{array}{ll}
  i^2=2i-12 \quad \quad & jk=2\bar{i}=2(2-i) \vspace{0.5ex}  \\
  j^2=-j-8 \quad \quad & ki=3\bar{j}=3(-1-j) \vspace{0.5ex}  \\
  k^2=-k-6 \quad \quad & ij=4\bar{k}=4(-1-k).
 \end{array}
$$
\end{example}

\section{Conclusion}
In this paper, we compute endomorphism rings of supersingular elliptic curves by Brandt--Sohn correspondence and Deuring correspondence. In particular, we give the relations between orientations of supersingular elliptic curves and coefficients of the corresponding ternary quadratic forms.

In general, the supersingular endomorphism ring problem is hard. We aim to find as many as possible supersingular elliptic curves whose endomorphism rings are easy to compute. Assume that an imaginary quadratic order $O$ can be embedded into $\text{End}(E)$.
We find that $\text{End}(E)$ is easy to compute if the absolute value of the discriminant of $O$ is a prime. Moreover, we give a basis of $\text{End}(E)$ if there exists a separable isogeny $\phi$ from $E$ to $E^p$ with prime degree.
It is surprising that the ternary quadratic form $F$ can represent double of $\deg(\phi)$ properly, where $F$ is the corresponding ternary quadratic form of $\text{End}(E)$.

For isogenous supersingular elliptic curves, we also study the relations of their endomorphism rings by ternary quadratic forms. As we know, the action of oriented isogenies can be represented by the composition of binary quadratic forms. Moreover, we study the actions of other isogenies and compute the ternary quadratic forms corresponding to the right orders of their kernel ideals.

We study the endomorphism ring of a supersingular elliptic curve as a maximal order in $B_{p,\infty}$, and give a basis of the maximal order. However, we do not give the explicit endomorphism maps of this basis. This is left as future work.

As we know, the isogeny-based cryptography is based on the hardness of the endomorphism ring problem. We believe that the correspondence between ternary quadratic forms and maximal orders in $B_{p,\infty}$ has a potential application in the isogeny-based cryptography. We will study the security of OSIDH and SQISign by this correspondence in future work.

%



\end{document}